\documentclass [a4paper,11pt]{amsart}

\usepackage[utf8]{inputenc}

\usepackage[pdftex]{graphicx} 
\DeclareGraphicsExtensions{.png,.pdf,.jpg}
\usepackage{amsmath}
\usepackage{amssymb}
\usepackage{latexsym}
\def\dim{\operatorname{dim}}

\newcommand{\C}{\mathbb{C}}

\newtheorem{teo}{Theorem}[section]

\newtheorem{co}[teo]{Corollary}

\theoremstyle{definition}
\newtheorem{defi}[teo]{Definition}

\theoremstyle{definition}
\newtheorem{ex}[teo]{Example}

\theoremstyle{definition}

\begin{document}

\title{A LÊ-GREUEL TYPE FORMULA FOR THE IMAGE MILNOR NUMBER}
\author{J.J. Nuño-Ballesteros, I. Pallarés-Torres}

\address{Departament de Matemàtiques,
Universitat de Val\`encia, Campus de Burjassot, 46100 Burjassot
SPAIN}

\email{Juan.Nuno@uv.es}

\email{irpato@alumni.uv.es}

\date{}

\thanks{Work partially supported by DGICYT Grant MTM2015-64013-P}

\begin{abstract} Let $f:(\C^n,0)\rightarrow (\C^{n+1},0)$ be a corank 1 finitely determined map germ. For a generic linear form $p:(\C^{n+1},0)\to(\C,0)$ we denote by $g:(\C^{n-1},0)\rightarrow (\C^{n},0)$ the transverse slice of $f$ with respect to $p$. We prove that the sum of the image Milnor numbers $\mu_I(f)+\mu_I(g)$ is equal to the number of critical points of the stratified Morse function $p|_{X_s}:X_s\to\C$, where $X_s$ is the disentanglement of $f$ (i.e., the image of a stabilisation $f_s$ of $f$).
\end{abstract}

\maketitle

\section{Introduction}

The Lê-Greuel formula \cite{GRE,LE} provides a recursive method to compute the Milnor number of an isolated complete intersection singularity (ICIS). We recall that if $(X,0)$ is a $d$-dimensional ICIS defined as the zero locus of a map germ $g:(\C^n,0)\rightarrow(\C^{n-d},0)$, then the Milnor fibre $X_s=g^{-1}(s)$ (where $s$ is a generic value in $\C^{n-d}$) has the homotopy type of a bouquet of $d$-spheres and the number of such spheres is called the Milnor number $\mu(X,0)$. If $d>0$, we can take $p:\C^n\rightarrow \C$ a generic linear projection with $H=p^{-1}(0)$ and such that $(X\cap H,0)$ is a $(d-1)$-dimensional ICIS.
Then,
\begin{equation}\label{Le-Greuel}
\mu(X,0)+\mu(X\cap H,0)=\dim_\C\frac{\mathcal{O}_n}{(g)+J(g,p)},
\end{equation}
where $\mathcal{O}_n$ is the ring of function germs from $(\C^n,0)$ to $\C$, $(g)$ is the ideal in $\mathcal{O}_n$ generated by the components of $g$ and $J(g,p)$ is the Jacobian ideal of $(g,p)$ (i.e., the ideal generated by the maximal minors of the Jacobian matrix). Note that $X_s$ is smooth and if $p$ is generic enough, then the restriction $p|_{X_s}:X_s\to\C$ is a Morse function and the dimension appearing in the right hand side of \eqref{Le-Greuel} is equal to the number of critical points of 
$p|_{X_s}$.

The aim of this paper is to obtain a Lê-Greuel type formula for the image Milnor number of a finitely determined map germ $f:(\C^n,0)\rightarrow (\C^{n+1},0)$. Mond showed in \cite{VC} that the disentanglement $X_s$ (i.e., the image of a stabilisation $f_s$ of $f$) has the homotopy type of a bouquet of $n$-spheres and the number of such spheres is called the image Milnor number $\mu_I(f,0)$.
The celebrated Mond's conjecture says that 
$$
\mathcal A_e\text{-}\mbox{codim}(f)\le \mu_I(f),
$$
with equality if $f$ is weighted homogeneous. Mond's conjecture is known to be true for $n=1,2$ but it remains still open for $n\ge 3$ (see \cite{VC,LB}). We feel that our Lê-Greuel type formula can be useful to find a proof of the conjecture in the general case. In fact, it would be enough to prove that the module which controls the number of critical points of a generic linear function is Cohen-Macaulay and then, use an induction argument on the dimension $n$ (see \cite{BNP} for details about Mond's conjecture).

We assume that $f$ has corank 1 and $n>1$. Then given a generic linear form $p:\C^{n+1}\rightarrow \C$ we can see $f$ as a 1-parameter unfolding of another map germ $g:(\C^{n-1},0)\rightarrow (\C^{n},0)$ which is the transverse slice of $f$ with respect to $p$. This means that $g$ has image $(X\cap H,0)$, where $(X,0)$ is the image of $f$ and $H=p^{-1}(0)$. The disentanglement $X_s$ is not smooth but it has a natural Whitney stratification given by the stable types. If $p$ is generic enough, the restriction $p|_{X_s}:X_s\to\C$ is a stratified Morse function. Our Lê-Greuel type formula is
\begin{equation}\label{Le-Greuel2}
\mu_I(f)+\mu_I(g)=\#\Sigma(p|_{X_s}),
\end{equation}
where the right hand side is the number of critical points of $p|_{X_s}$ as a stratified Morse function. The case $n=1$ has to be considered separately, in this case we have
\begin{equation}\label{Le-Greuel3}
\mu_I(f)+m_0(f)-1=\#\Sigma(p|_{X_s}),
\end{equation}
where $m_0(f)$ is the multiplicity of the curve parametrized by $f$. This makes sense, since $\mu(X,0)=m_0(X,0)-1$ for a  0-dimensional ICIS $(X,0)$.

\section{Multiple point spaces and Marar's formula}


In this section we recall Marar's formula for the Euler characteristic of the disentanglement of a corank 1 finitely determined map germ. We first recall the Marar-Mond \cite{MM} construction of the $k$th-multiple point spaces for corank 1 map germs, which is based on the iterated divided differences. Let $f:(\C^n,0)\rightarrow (\C^{p},0)$ be a corank 1 map germ. We can choose coordinates in the source and target such that $f$ is written in the following form:
$$
f(x,z)=(x,f_n(x,z),\dots,f_p(x,z)),\ x\in\C^{n-1},\ z\in\C.
$$
This forces that if $f(x_1,z_1)=f(x_2,z_2)$ then necessarily $x_1=x_2$. Thus, it makes sense to embed the double point space of $f$ in $\C^{n+1}$ instead of $\C^n\times\C^n$. Analogously, we will consider the $k$th-multiple point space embedded in $\C^{n+k-1}$.

We construct an ideal $I_k(f)\subset \mathcal{O}_{n+k-1}$ defined as follows: $I_k(f)$ is generated by $(k-1)(p-n+1)$ functions $\Delta_i^{(j)}\in  \mathcal{O}_{n+k-1}$, $1\le i\le k-1$, $n\le j\le p$. Each $\Delta_i^{(j)}$ is a function only on the variables $x,z_1,\dots,z_{i+1}$ such that:
$$
\Delta^{(j)}_1(x,z_1,z_2)=\frac{f_j(x,z_1)-f_j(x,z_2)}{z_1-z_2},
$$
and for $1\le i\le k-2$,
$$
\Delta^{(j)}_{i+1}(x,z_1,\dots,z_{i+2})=\frac{\Delta^{(j)}_i(x,z_1,\dots,z_i,z_{i+1})-\Delta^{(j)}_i(x,z_1,\dots,z_i,z_{i+2})}{z_{i+1}-z_{i+2}}.
$$

\begin{defi} 
The \emph{$k$th-multiple point space} is $D^k(f)=V(I_k(f))$, the zero locus in $(\C^{n+k-1},0)$ of the ideal $I_k(f)$.
\end{defi}

(We remark that the $k$th-multiple point space is denoted by $\widetilde D^k(f)$ instead of $D^k(f)$ in \cite{MM}).

If $f$ is stable, then, set-theoretically, $D^k(f)$ is the Zariski closure of the set of points
$(x,z_1,\ldots,z_{k})\in \C^{n+k-1}$ such that:
$$f(x,z_1)=\ldots=f(x,z_k),\quad z_i\neq z_j, \text{ for } i\neq j,$$
(see \cite{MM,GJ}).
But, in general, this may be not true if $f$ is not stable. For instance, consider the cusp $f:(\C,0)\to(\C^2,0)$ given by $f(z)=(z^2,z^3)$. Since $f$ is one-to-one, the closure of the double point set is empty, but
$$
D^2(f)=V(z_1+z_2,z_1^2+z_1z_2+z_2^2).
$$
This example also shows that the $k$th-multiple point space may be non-reduced in general.

The main result of Marar-Mond in \cite{MM} is that the $k$th-multiple point spaces can be used to characterize the stability and the finite determinacy of $f$.

%
%

\begin{teo}\label{ICIS}\cite[2.12]{MM} Let $f:(\C^n,0)\rightarrow(\C^{p},0)$ ($n<p$) be a finitely determined map germ of corank 1. Then: 
\begin{enumerate}
\item $f$ is stable if and only if $D^k(f)$ is smooth of dimension $p-k(p-n)$, or empty, for $k\geq 2$. 
\item $f$ is finitely determined if and only if for each $k$ with $p-k(p-n)\geq 0$, $D^k(f)$ is either an ICIS of dimension $p-k(p-n)$ or empty, and if, for those $k$ such that $p-k(p-n)<0$, $D^k(f)$ consists at most of the point $\{0\}$.
\end{enumerate}
\end{teo}

The following construction is also due to Marar-Mond \cite{MM} and gives a refinement of the types of multiple points.

\begin{defi}
Let $\mathcal{P}=(r_1,\ldots,r_m)$ be a partition of $k$ (that is, $r_1+\ldots+r_m=k$). Let $I(\mathcal{P})$ be the ideal in $\mathcal{O}_{n-1+k}$ generated by the $k-m$ elements $z_i-z_{i+1}$ for $r_1,\ldots,r_{j-1}+1\leq i\leq r_1,\ldots,r_{j}-1$, $r\leq j\leq m$.
Define the ideal $I_k(f,\mathcal{P})=I_k(f)+I(\mathcal{P})$ and the \emph{$k$-multiple point space} of $f$ with respect to the partition $\mathcal{P}$ as $D^k(f,\mathcal{P})=V(I_k(f,\mathcal{P}))$.
\end{defi}

\begin{defi}
We define a \emph{generic point} of $D^k(f,\mathcal{P})$
as a point 
$$
(x,z_1,\dots,z_1,\dots,z_m,\dots,z_m),
$$ 
($z_i$ iterated $r_i$ times, and $z_i\ne z_j$ if $i\ne j$) such that the local algebra of $f$ at $(x,z_i)$ is isomorphic to $\C[t]/(t^{r_i})$, and such that
$$f(x,z_1)=\ldots=f(x,z_m).
$$
\end{defi}

If $f$ is stable, then $D^k(f,\mathcal{P})$ is equal to the Zariski closure of its generic points
(see \cite{MM}). Moreover, we have the following corollary, which extends Theorem \ref{ICIS} to the multiple point spaces with respect to the partitions.

\begin{co}\label{ICIS2}\cite[2.15]{MM} If $f$ is finitely determined (resp. stable), then for each partition $\mathcal{P}=(r_1,\ldots,r_m)$ of $k$ satisfying $p-k(p-n+1)+m\geq 0$, the germ of $D^k(f,\mathcal{P})$ at $\{0\}$ is either an ICIS (resp. smooth) of dimension $p-k(p-n+1)+m$, or empty. Moreover, those $D^k(f,\mathcal{P})$ for $\mathcal{P}$ not satisfying the inequality consist at most of the single point $\{0\}$.
\end{co}

The next step is to observe that the $k$th-multiple point space $D^k(f)$ is invariant under the action of the $k$th symmetric group $S_k$.

\begin{defi}
Let $M$ be a $\mathbb{Q}$-vector space upon which $S_k$ acts. Then the \emph{alternating part} of $M$, denoted by $\text{Alt}_k M$, is defined to be 
$$\text{Alt}_k M:=\{m\in M : \sigma(m)=\mbox{sign}(\sigma)m, \text{ for all } \sigma\in S_k\}.$$
\end{defi}

The following theorem of Goryunov-Mond in \cite{GM} allows us to compute the image Milnor number of $f$ by means of a spectral sequence associated to the multiple point spaces. 

\begin{teo}\cite[2.6]{GM} \label{Hn} Let $f:(\C^{n},0)\rightarrow (\C^{n+1},0)$ be a corank 1 map germ and $f_s$ a stabilisation of $f$, for $s\neq 0$ and $X_s$ the image of $f_s$. Then,
$$H_n(X_s,\mathbb{Q})\cong \bigoplus_{k=2}^{n+1} \mbox{\emph{Alt}}_k(H_{n-k+1}(D^k(f_s),\mathbb{Q})).$$
\end{teo}

Note that since $X_s$ has the homotopy type of a wedge of $n$-spheres, the image Milnor number of $f$ is the rank of $H_n(X_s,\mathbb{Q})$. If we consider $H_n(X_s,\mathbb{Q})$ as a $\mathbb{Q}$-vector space, 
$$\mu_I(f)=\dim_{\mathbb{Q}}H_n(X_s,\mathbb{Q}).
$$
So, by Theorem \ref{Hn}, the image Milnor number is 
$$
\mu_I(f)=\sum_{k=2}^{n+1} \dim_{\mathbb{Q}}\mbox{Alt}_k(H_{n-k+1}(D^k(f_s),\mathbb{Q})).
$$
By \cite[Corollary 2.8]{H}, we can compute the alternating Euler characteristic of $D^k(f_s)$ as follows: for each partition $\mathcal P$, we set
$$\beta(\mathcal P)=\frac{\text{sign}(\mathcal{P})}{\prod_{i}\alpha_i!{i}^{\alpha_i}} ,
$$
where $\alpha_i:=\#\{j:\ r_j=i\}$ and 
$\mbox{sign}(\mathcal{P})$ is the number $(-1)^{k-\sum_{i}\alpha_i}$. Then, 
$$\chi^{alt}(D^k(f_s))=\sum_{|\mathcal{P}|=k}\beta(\mathcal{P})\chi(D^k(f_s,\mathcal{P})).$$ 

Moreover, by Theorem \ref{ICIS} and Corollary \ref{ICIS2}, $D^k(f_s)$ (resp. $D^k(f_s,\mathcal{P})$) is a Milnor fibre of the ICIS $D^k(f)$ (resp. $D^k(f,\mathcal{P})$), and hence it has the homotopy type of a wedge of spheres of real dimension $\dim D^k(f)=n-k+1$ (resp. $\dim D^k(f,\mathcal{P})$). Thus, 
$$
\dim_{\mathbb{Q}}\mbox{Alt}_k(H_{n-k+1}(D^k(f_s),\mathbb{Q}))=(-1)^{n-k+1} \chi^{alt}(D^k(f_s)),
$$
and
$$\chi(D^k(f_s,\mathcal{P}))=1+(-1)^{\dim D^k(f,\mathcal{P})}\mu(D^k(f,\mathcal{P})).$$
This gives the following strong version of Marar's formula \cite{M}:
\begin{equation}
\mu_I(f)=\sum_{k=2}^{n+1}(-1)^{n-k+1}\sum_{|\mathcal{P}|=k}\beta(\mathcal P)(1+(-1)^{\dim D^k(f,\mathcal{P})}\mu(D^k(f,\mathcal{P}))),\label{eq}\end{equation}
where the coefficients $\beta(\mathcal{P})=0$ when the sets $D^k(f,\mathcal{P})$ are empty, for $k=2,\ldots,n+1$.

\section{Lê-Greuel type formula}

Let $f:(\C^n,0)\rightarrow(\C^{n+1},0)$ be a corank 1 finitely determined map germ. 
Let $p:\C^{n+1}\to\C$ be a generic linear projection such that $H=p^{-1}(0)$ is a generic hyperplane through the origin in $\C^{n+1}$. We can choose linear coordinates in $\C^{n+1}$ such that $p(y_1,\dots,y_{n+1})=y_1$. Then, we choose the coordinates in $\C^n$ in such a way that $f$ is written in the form
$$f(x_1,\ldots,x_{n-1},z)=(x_1,\ldots,x_{n-1},h_1(x_1,\ldots,x_{n-1},z),h_2(x_1,\ldots,x_{n-1},z)),$$
for some holomorphic functions $h_1,h_2$. We see $f$ as a 1-parameter unfolding of the map germ $g:(\C^{n-1},0)\rightarrow(\C^{n},0)$ given by 
$$
g(x_2,\ldots,x_{n-1},z)=(x_2,\ldots,x_{n-1},h_1(0,x_2,\ldots,x_{n-1},z),h_2(0,x_2,\ldots,x_{n-1},z)).
$$
We say that $g$ is the transverse slice of $f$ with respect to the generic hyperplane $H$. If $f$ has image $(X,0)$ in $(\C^{n+1},0)$, then the image of $g$ in $(\C^{n},0)$ is isomorphic to $(X\cap H,0)$.

We take $f_s$ a stabilisation of $f$ and denote by $X_s$ the image of $f_s$ (see \cite{VC} for the definition of stabilisation). Since $f$ has corank 1, $X_s$ has a natural Whitney stratification given by the stable types of $f_s$. In fact, the strata are the submanifolds 
$$
M^k(f_s,\mathcal P):=\epsilon^k(D^k(f_s,\mathcal P)^0)\setminus \epsilon^{k+1}(D^{k+1}(f_s)),
$$
where $D^k(f_s,\mathcal P)^0$ is the set of generic points of $D^k(f_s,\mathcal P)$, $\epsilon^k:\C^{n+k-1}\to\C^{n+1}$ is the map $(x,z_1,\dots,z_k)\mapsto f_s(x,z_1)$ and $\mathcal P$ runs through all the partitions of $k$ with $k=2,\dots,n+1$. We can choose the generic linear projection $p:\C^{n+1}\to\C$ in such a way that the restriction to each stratum $M^k(f_s,\mathcal P)$ is a Morse function. In other words, such that the restriction $p|_{X_s}:X_s\to \C$ is a stratified Morse function in the sense of \cite{GorMac}. We will denote by $\#\Sigma(p{|_{X_s}})$ the number of critical points as a stratified Morse function. Our first result in this section is for the case of a plane curve.

\begin{teo}\label{formula1}
Let $f:(\C,0)\rightarrow(\C^{2},0)$ be a corank 1 finitely determined map germ. Let $p:\C^{2}\rightarrow \C$ be a generic linear projection, then  
$$
\#\Sigma(p{|_{X_s}})=\mu_{I}(f)+m_0(f)-1,
$$
where $m_0(f)$ is the multiplicity of $f$.
\end{teo}

\begin{proof}
After a change of coordinates, we can assume that
$$f(t)=(t^k,c_mt^m+c_{m+1}t^{m+1}+\dots),
$$ 
where $k=m_0(f)$, $m>k$ and $c_m\ne0$. The stabilisation $f_s$ is an immersion with only transverse double points. So, its image $X_s$ has only two strata: $M^2(f_s,(1,1))$ is a $0$-dimensional stratum composed by the transverse double points and $M^1(f_s,(1))$ is a $1$-dimensional stratum given by the smooth points of $X_s$. Note that the number of double points of $f_s$ is the delta invariant of the plane curve, $\delta(X,0)$, which is equal to $\mu_I(f)$ by  \cite[Theorem 2.3]{LB}.

Let $p:\C^{2}\rightarrow \C$ be a generic linear projection such that $p|_{X_s}$ is a stratified Morse function. 
Then: 
$$\#\Sigma(p{|_{X_s}})=\#M^2(f_s,(1,1))+\#\Sigma(p{|_{M^1(f_s,(1))}})=
\mu_I(f)+\#\Sigma(p{|_{M^1(f_s,(1))}}).
$$
Since $f_s$ is a diffeomorphism on the stratum $M^1(f_s,(1))$, the number of critical points of $p{|_{M^1(f_s,(1))}}$ is equal to the number of critical points of $p\circ f_s$. Assume that $p(x,y)=Ax+By$ with $A\ne0$. Then $p\circ f_s$ is a Morsification of the function
$$
p\circ f(t)=A t^k+B(c_mt^m+c_{m+1}t^{m+1}+\ldots)
$$
The number of critical points of $p\circ f_s$ is equal to $\mu(p\circ f)=k-1=m_0(f)-1$, which proves our formula.
\end{proof}

Next, we state and prove the formula for the case $n>1$.

\begin{teo}\label{formula}
Let $f:(\C^n,0)\rightarrow(\C^{n+1},0)$ be a corank 1 finitely determined map germ with $n>1$. Let $p:\C^{n+1}\rightarrow \C$ be a generic linear projection which defines a transverse slice $g:(\C^{n-1},0)\rightarrow(\C^n,0)$. Then,
$$
\#\Sigma(p{|_{X_s}})=\mu_{I}(f)+\mu_I(g).
$$
\end{teo}

\begin{proof}
By Marar's formula (\ref{eq}):
\begin{align*}
\mu_I(f)+\mu_I(g)&=\sum_{k=2}^{n+1}(-1)^{n-k+1}\sum_{|\mathcal{P}|=k}\beta(\mathcal P)(1+(-1)^{\dim D^k(f,\mathcal{P})}\mu(D^k(f,\mathcal{P})))\\
&+\sum_{k=2}^{n}(-1)^{n-k}\sum_{|\mathcal{P}|=k}\beta(\mathcal P)(1+(-1)^{\dim D^k(g,\mathcal{P})}\mu(D^k(g,\mathcal{P})))\\
\end{align*}
Note that if $\dim D^k(f,\mathcal{P})>0$, then $\dim D^k(f,\mathcal{P})=1+\dim D^k(g,\mathcal{P})$. Otherwise, if $\dim D^k(f,\mathcal{P})=0$, then $D^k(g,\mathcal{P})=\emptyset$. So, we can separate the formula into two parts, the first one for partitions with $\dim D^k(f,\mathcal{P})=0$, the second one for partitions with $\dim D^k(f,\mathcal{P})>0$. Thus,
\begin{align*}
&\mu_I(f)+\mu_I(g)=\sum_{k=2}^{n+1}(-1)^{n+k-1}\sum_{\substack{|\mathcal{P}|=k \\ \dim D^k(f,\mathcal{P})=0}}\beta(\mathcal P)(1+\mu(D^k(f,\mathcal{P})))\\
&+\sum_{k=2}^{n}(-1)^{n+k-1}\sum_{\substack{|\mathcal{P}|=k \\ \dim D^k(f,\mathcal{P})>0}}\beta(\mathcal P)(-1)^{\dim D^k(f,\mathcal{P})}(\mu(D^k(f,\mathcal{P}))+\mu(D^k(g,\mathcal{P})))
\end{align*}
If $\dim D^k(f,\mathcal{P})=0$, the Milnor number of $D^k(f,\mathcal{P})$ is $$\mu(D^k(f,\mathcal{P}))=deg(D^k(f,\mathcal{P}))-1,$$ where $deg$ is the degree of the map germ that defines the 0-dimensional ICIS $D^k(f,\mathcal{P})$. 

We choose the coordinates such that $p(y_1,\dots,y_{n+1})=y_{1}$. We denote by $\tilde{p}:\C^{n+k-1}\rightarrow \C$ the projection onto the first coordinate. Then:
$$
D^k(g,\mathcal{P})=D^k(f,\mathcal{P})\cap \tilde p^{-1}(0).
$$
By the Lê-Greuel formula for ICIS \cite{GRE,LE},
$$
\mu(D^k(f,\mathcal{P}))+\mu(D^k(g,\mathcal{P}))=\# \Sigma(\tilde{p}{|_{D^k(f_s,\mathcal{P})}}).
$$

It is easy to check that $(-1)^{\dim D^k(f)} \text{sign}(\mathcal{P})(-1)^{\dim D^k(f,\mathcal{P})}=1$ for any partition $\mathcal P$. Thus, we get:
\begin{align*}
\mu_I(f)+\mu_I(g)&=\sum_{k=2}^{n+1}\sum_{\substack{|\mathcal{P}|=k \\ \dim D^k(f,\mathcal{P})=0}}\frac{deg(D^k(f,\mathcal{P}))}{\prod_{i}\alpha_i!{i}^{\alpha_i}}\\
&+\sum_{k=2}^{n}\sum_{\substack{|\mathcal{P}|=k \\ \dim D^k(f,\mathcal{P})>0}}\frac{\# \Sigma(\tilde{p}{|_{D^k(f_s,\mathcal{P})}})}{\prod_{i}\alpha_i!{i}^{\alpha_i}}.\\
\end{align*}

For any partition $\mathcal P$, the map $\epsilon^k:D^k(f_s,\mathcal{P})^0\to \C^{n+1}$ is a covering map of degree $\prod_{i}\alpha_i!{i}^{\alpha_i}$ onto its image. If $\dim D^k(f,\mathcal{P})=0$, then all the points of $D^k(f_s,\mathcal{P})$ are generic and moreover, $\epsilon^k(D^k(f_s,\mathcal{P}))$ does not contain points of $\epsilon^{k+1}(D^{k+1}(f_s))$, so
$$
deg(D^k(f,\mathcal{P}))=\# D^k(f_s,\mathcal P)=(\prod_{i}\alpha_i!{i}^{\alpha_i})\# M^k(f_s,\mathcal P).
$$
Otherwise, if $\dim D^k(f,\mathcal{P})>0$, then by the genericity of the projection we can assume that all the critical points of $\tilde{p}{|_{D^k(f_s,\mathcal{P})}}$ are generic points of $D^k(f_s,\mathcal{P})$ and whose image is contained in $M^k(f_s,\mathcal P)$, hence:
$$
\# \Sigma(\tilde{p}{|_{D^k(f_s,\mathcal{P})}})
=(\prod_{i}\alpha_i!{i}^{\alpha_i})\# \Sigma(p{|_{M^k(f_s,\mathcal P)}}).
$$

Thus, we conclude that
\begin{align*}
\mu_I(f)+\mu_I(g)&=\sum_{k=2}^{n+1}\sum_{\substack{|\mathcal{P}|=k \\ \dim D^k(f,\mathcal{P})=0}}\# M^k(f_s,\mathcal{P})\\
&+\sum_{k=2}^{n}\sum_{\substack{|\mathcal{P}|=k, \\ \dim D^k(f,\mathcal{P})>0}}\# \Sigma(p{|_{M^k(f_s,\mathcal{P})}}),
\end{align*}
which is nothing but the number of critical points of the stratified Morse function $p|_{X_s}$.
\end{proof}

\section{Examples}
In this section, we give some examples to illustrate the formulas of theorems \ref{formula1} and \ref{formula}.

\begin{ex} (The singular plane curve $E_6$)


 Let $f(z)=(z^3,z^4)$ be the singular plane curve $E_6$, let $f_s(z)=(z^3 +  sz, z^4 + \frac{5}{4} s z^2)$ be a stabilisation of $f$, for $s\neq 0$.

 \begin{figure}[ht]
\centering
\includegraphics[scale=0.4]{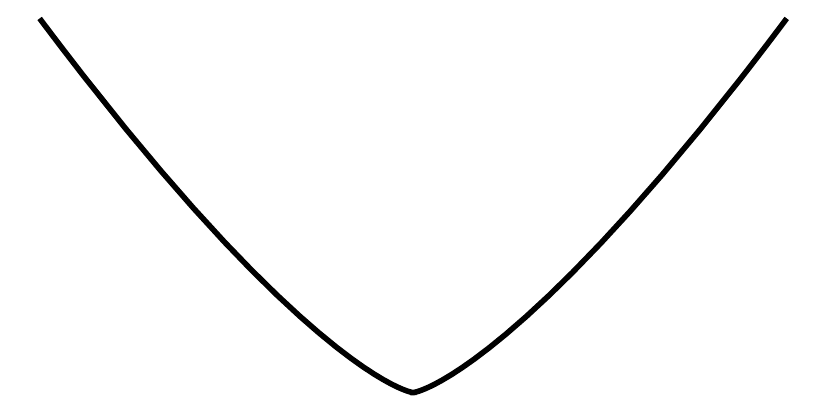}\hspace{1cm}\includegraphics[scale=0.4]{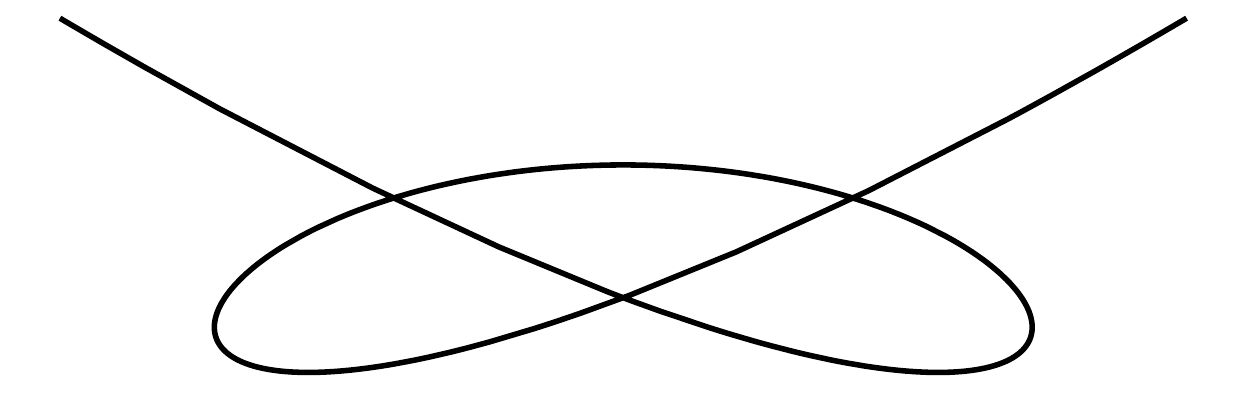}
\caption{The curve $E_6$ and its stabilisation for $s<0$}
\end{figure}

Let $M^2(f_s,(1,1))$ be the $0$-dimensional stratum of $X_s$. It is composed by three points, they correspond to three double transversal points. Let $M^1(f_s,(1))$ be the $1$-dimensional stratum. If we compose $f_s$ with $p(z,u)=z$ there are two critical points in a neighbourhood of the origin, so $\sharp\sum p_{|_{X_s}}=5.$

\begin{figure}[ht]
\centering
\includegraphics[scale=0.4]{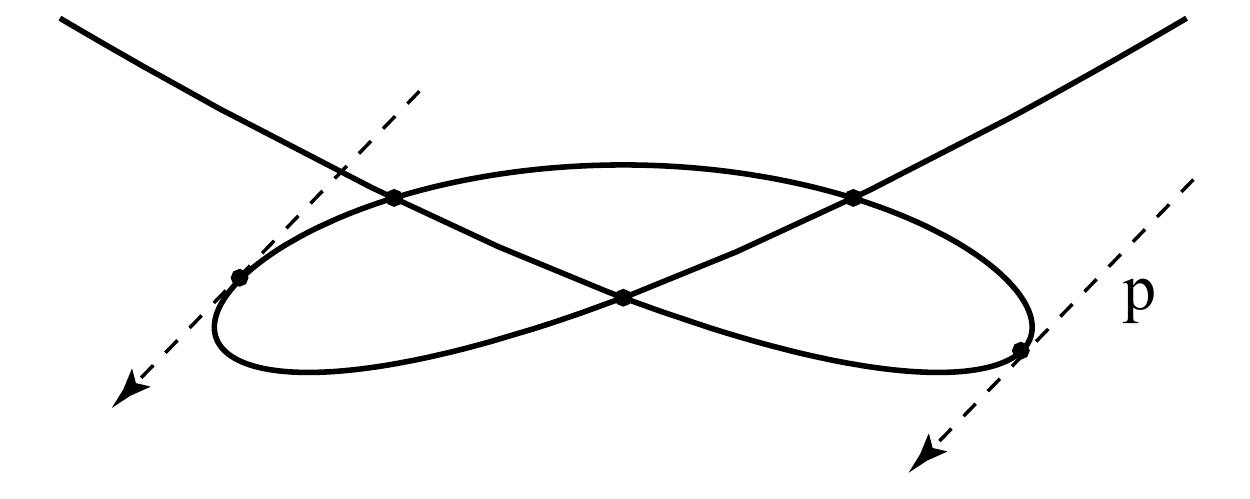}
\caption{Critical points in $X_s$}
\end{figure}

Now, since the multiplicity of $f$, $m_0(f)=3$ and the image Milnor number of $f$ is $\mu_I(f)=3$, $\mu_I(f)+m_0(f)-1=5$. We conclude that the formula is true.
\end{ex}

When $n>1$, we proceed in the following way:
Let $f:(\C^n,0)\rightarrow (\C^{n+1},0)$ be a corank 1 finitely determined map germ written as
$$f(x,z)=(x,h_1(x,z),h_2(x,z)),\ x\in\C^{n-1},\ z\in\C.
$$ 
Let $f_s$ be a stabilisation of $f$. The image of $f_s$ is denoted by $X_s$. First, we calculate the number of critical points of the restriction of $p$ to $X_s$, for the generic linear projection $p(y_1,\dots,y_{n+1})=y_1$. We separate the image set $X_s$ in strata of different dimesnions given by stable types, which correspond to the sets $M^k(f_s,\mathcal{P})$. The $n$-dimensional stratum, $M^1(f_s,(1))$, is composed by the regular part of $f_s$. So, the restriction $p{|_{M^1(f_s)}}$ has not critical points.  

The $(n-1)$-dimensional stratum is composed by $M^2(f_s,(1,1))$. To calculate the critical points, we will work with the inverse image by $\epsilon^2$, that is, $D^2(f_s,(1,1))=D^2(f_s)$. The double point space $D^2(f_s)$ is a subset of $\C^{n+1}$, but we take a projection of $D^2(f_s)$ in the first $n$ variables. So, we denote by $D(f_s)$ the projection of double point space in $\C^n$.
The double point space $D(f_s)$ is a hypersurface in $\C^n$ given by the resultant of $P_s$ and $Q_s$ with respect to $z_2$, where $P_s= \frac{h_{1,s}(x,z_2)-h_{1,s}(x,z_1)}{z_2-z_1}$ and $Q_s= \frac{h_{2,s}(x,z_2)-h_{2,s}(x,z_1)}{z_2-z_1}$. This gives the defining equation of $D(f_s)$, denoted by $\lambda_s(x,z)=0$.

To calculate the critical points of the set $D(f_s)$ we take the linear projection $\tilde{p}(x_1,\ldots,x_{n-1},z)=x_1$. 
Note that the hypersuface $D(f_s)$ also contains the critical points of the other $k$-dimensional strata, with $k<n-1$. Then, it will be sufficient to compute critical points here, in order to have all the critical points. 
By definition, we say that $(x_1,\ldots,x_{n-1},z)$ is a critical point of $\tilde{p}_{|_{D(f_s)}}$ if $\lambda_s(x,z)=0$ and $J(\lambda_s,\tilde{p})(x,z)=0$, where $J(\lambda_s,\tilde{p})$ is the Jacobian determinant of $\lambda$ and $\tilde{p}$.

If a critical point of $\tilde{p}_{|_{D(f_s)}}$ corresponds to a $m$-multiple point, then we will 
have  $m$ critical points in $D(f_s)$ for one in the image of $f_s$. 
Thus, once the critical points of each type are obtained, we have to divide by the multiplicity of the point. In this way, we obtain the number of critical points of $p$ in the image of $f_s$.

On the other hand, we compute separately the image Milnor numbers of $f$ and $g$ in order to check the formulas.

\begin{ex} (The germ $F_4$ in $\C^3$)
Let $f(x,z)=(x,z^2, z^5 + x^3 z)$ be the germ $F_4$. Let $f_s(x,z)=(x, z^2, z^5 + x s z^3 + (x^3 - 5 x s - s) z)$ be a stabilisation of $f$, for $s\neq0.$ By \cite{SC}, $f$ is a 1-parameter unfolding of the plane curve $A_4,$ $g(z)=(z^2,z^5)$ and in fact, $g$ is the transverse slice of $f$.

\begin{figure}[ht]
\centering
\includegraphics[scale=0.4]{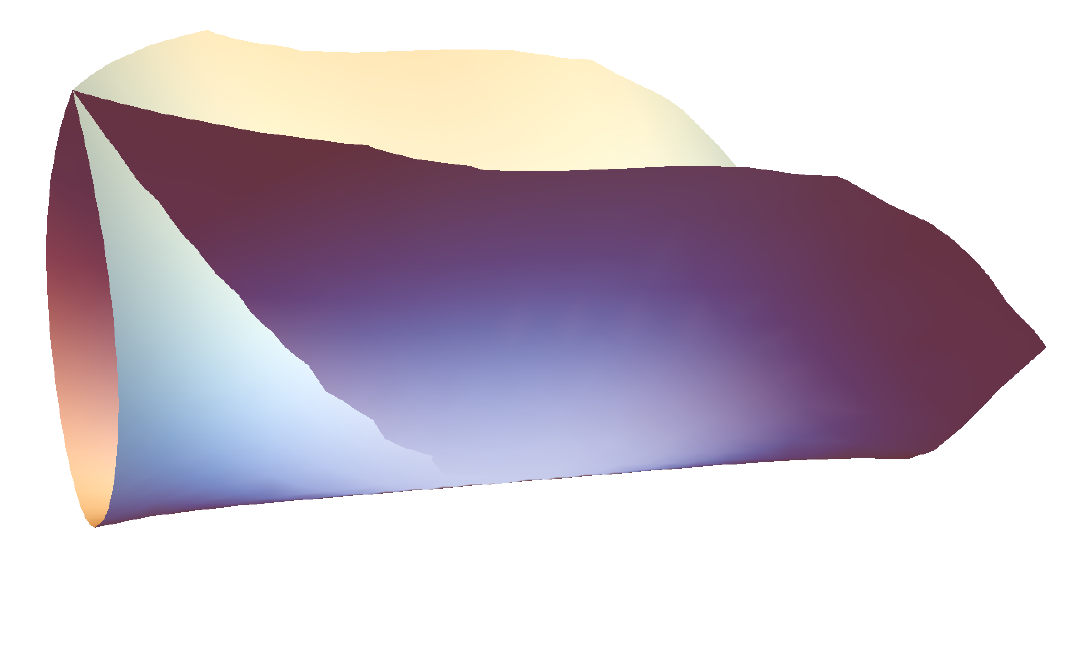}\hspace{1cm}
\includegraphics[scale=0.4]{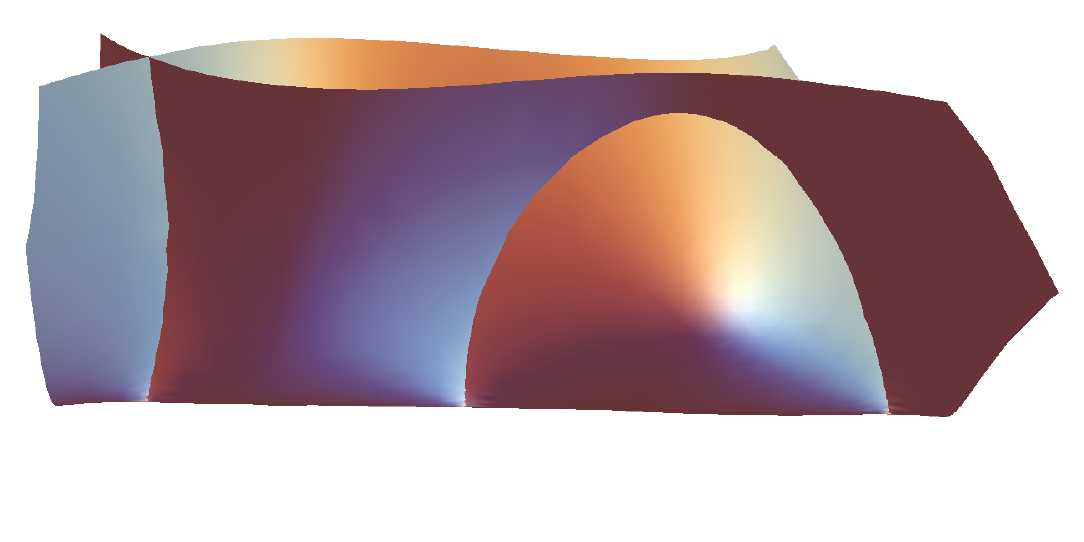}
\caption{The germ $F_4$ and its stabilisation for $s>0$}
\end{figure}

Let $M^3(f_s,(1,1,1))\cup M^2(f_s,(2))$ be the 0-dimensional strata of $X_s$. In our case, there are no triple points and they appear three cross caps in $M^2(f_s,(2))$. 

Let $M^2(f_s,(1,1))$ be the $1$-dimensional stratum of $X_s$. As we said before, let $D^2(f_s)$ be the double point curve in $\C^3$, and by projecting in the first two coordinates, we have the double point curve in $\C^2$, denote by $D(f_s)$.


We compute the resultant of $P_s$ and $Q_s$ respect to $z_2$, where $P_s$ and $Q_s$ are the divided differences. The double point curve of $f_s$ in $\C^2$ is the plane curve $$\lambda_s(x,z)=-s - 5 s x + x^3 + s x z^2 + z^4.$$  
The critical points of the restriction $p{|_D(f_s)}$ are given by
$\lambda_s(x_0,z_0)=0$ and $J(\lambda_s,\tilde{p})(x_0,z_0)=0$, where $\tilde p(x,z)=x$.

\begin{figure}[ht]
\centering
\includegraphics[scale=0.5]{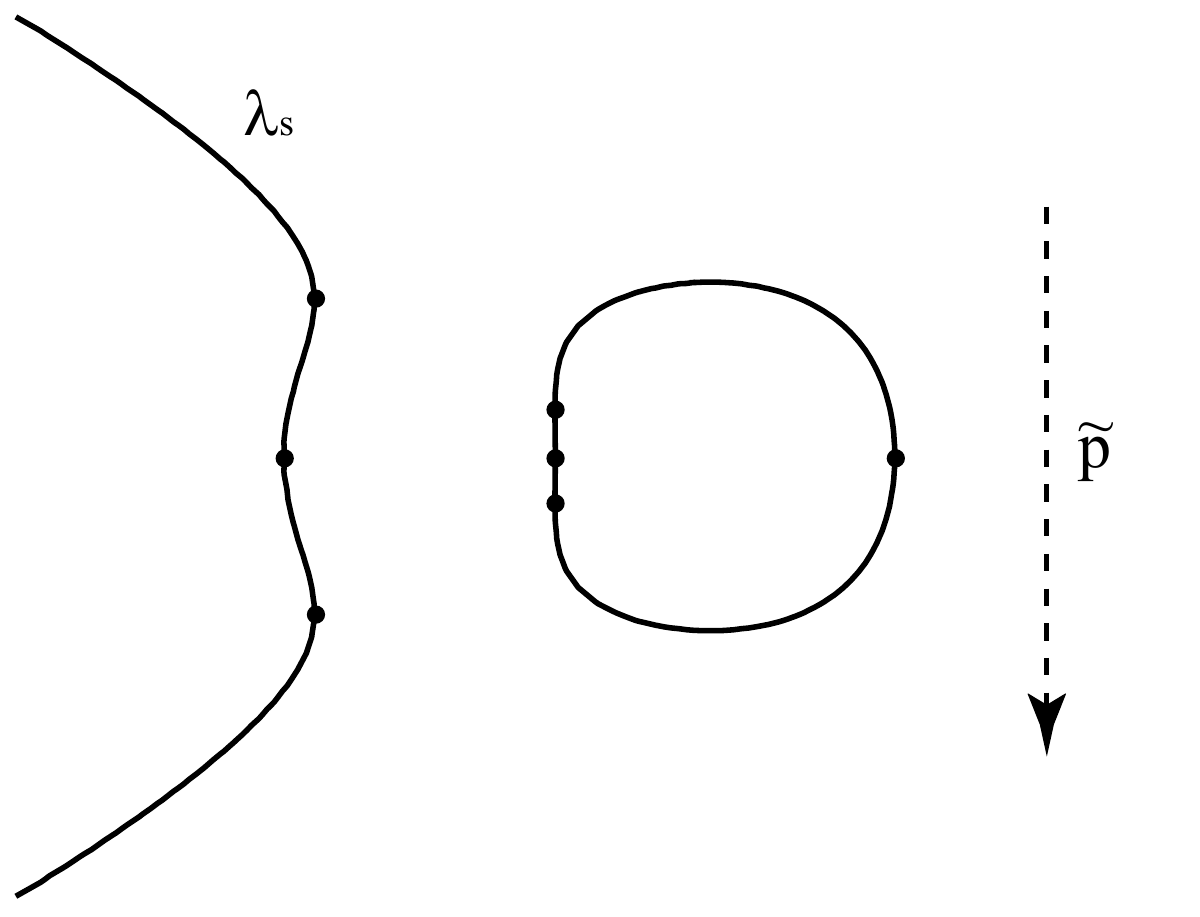}
\caption{Cusps and tacnodes in the double point curve}\label{doblesF4}
\end{figure}

Nine critical points are obtained. Three of these points are cusps in $g_{x,s}$, 
which correspond to the three cross caps of $f_s$. 
Then, the other six critical points in
$\tilde{p}_{|_{\lambda_s(x_0,z_0)=0}}$ correspond to three tacnodes in $g_{x,s}$, which are represented in the double point curve when a vertical line is tangent at two points of $D(f_s)$. So, each two of these critical points in $\lambda_s$ correspond to one tacnode of $g_{x,s}$ in $M^2(f_s,(1,1))$. Note that in the Fig. \ref{doblesF4} there are only two tacnodes, that is because the other is a complex tacnode.

Finally, in the $2$-dimensional stratum $M^1(f_s,(1))$ there are no  critical points. So, the number of critical points in $X_s$ is $\#\Sigma p{|_{X_s}}=6,$ three cusps, three tacnodes and zero triple points. Then,  $\#\Sigma p{|_{X_s}}=C+J+T$ where $C,J,T$ are the numbers of cusps, tacnodes and triple points respectively of $g_{x,s}$. 
By \cite{SC}, $\mu_I(f)=C+J+T-\delta(g)$. Since $g$ is a plane curve, we have that $\mu_I(g)=\delta(g)$ (see \cite{LB}). So, $$\#\Sigma p{|_{X_s}}=C+J+T=\mu_I(f)+\mu_I(g).$$ 
\end{ex}

\bibliographystyle{amsplain}
\bibliography{artLeGreuel}

\end{document}